\documentclass[12pt]{article}
\usepackage{amsmath,amssymb,amsfonts,amsthm}
\newenvironment{myabstract}{\par\noindent
{\bf Abstract . } \small }
{\par\vskip8pt minus3pt\rm}
\newcounter{item}[section]
\newcounter{kirshr}
\newcounter{kirsha}
\newcounter{kirshb}
\newenvironment{enumroman}{\setcounter{kirshr}{1}
\begin{list}{(\roman{kirshr})}{\usecounter{kirshr}} }{\end{list}}
\newenvironment{enumarab}{\setcounter{kirshb}{1}
\begin{list}{(\arabic{kirshb})}{\usecounter{kirshb}} }{\end{list}}
\newenvironment{athm}[1]{\vskip3mm\par\noindent
{\bf #1 }. \slshape }
{\upshape\par\vskip10pt minus3pt}
\newtheorem{theorem}{Theorem}[section]

\newtheorem{corollary}[theorem]{Corollary}
\newenvironment{demo}[1]{\noindent{\bf #1.}\upshape\mdseries}
{\nopagebreak{\hfill\rule{2mm}{2mm}\nopagebreak}\par\normalfont}
\theoremstyle{definition}

\newtheorem{definition}[theorem]{Definition}

\def\C{{\mathfrak{C}}}
\def\Fm{{\mathfrak{Fm}}}

\def\Nr{{\mathfrak{Nr}}}
\def\Fr{{\mathfrak{Fr}}}
\def\Sg{{\mathfrak{Sg}}}
\def\Fm{{\mathfrak{Fm}}}
\def\A{{\mathfrak{A}}}
\def\B{{\mathfrak{B}}}
\def\C{{\mathfrak{C}}}
\def\D{{\mathfrak{D}}}
\def\M{{\mathfrak{M}}}
\def\N{{\mathfrak{N}}}

\def\CA{{\bf CA}}

\def\Lf{{\bf Lf}}

\def\RCA{{\bf RCA}}

\def\Rd{{\ Rd}}
\def\(R)RA{{\bf (R)RA}}
\def\RA{{\bf RA}}

\def\QRA{{\sf QRA}}
\def\tr{{\sf tr}}

 \def\CA{{\sf CA}}
\def\B{{\sf B}}

\def\tp{{\sf tp}}
 
\def\Nr{{\mathfrak{Nr}}}

\def\Ra{{\mathfrak{Ra}}}

\def\Ra{{\mathfrak{Ra}}}
\def\Nr{{\mathfrak{Nr}}}

\def\A{{\mathfrak{A}}}
\def\B{{\mathfrak{B}}}
\def\C{{\mathfrak{C}}}
\def\D{{\mathfrak{D}}}

\def\A{{\mathfrak{A}}}
\def\B{{\mathfrak{B}}}
\def\C{{\mathfrak{C}}}
\def\D{{\mathfrak{D}}}

\def\Alg{{\mathfrak{Alg}}}

\def\L{{\mathfrak{L}}}

\def\L{{\mathfrak{L}}}
\def\CA{{\bf CA}}
\def\RA{{\bf RA}}

\def\RCA{{\bf RCA}}

\def\Diag{{\bf Diag}}
\def\Fl{{\mathfrak{Fl}}}

\def\CM{{\bf CM}}

\title{Quasi-projective relation algebras and directed cylindric algebras of any dimension are categorially 
equivalent}
\author{Tarek Sayed Ahmed}

\begin{document}
\maketitle

\begin{myabstract}{ We show that the class of relations with quasi projections $\QRA$  and Nemeti's directed cylindric algebras
$\CA^{\uparrow}$ are categorially equivalent.
There exists a functor from the former to the latter that is strongly invertible. We also prove that such algebras enjoy the superamamalgmation property.
Using pairing functions, stimulated by quasi-projection,  we formulate and prove a G\"odels second incompleteness theorem for
finite variable fragments, and we discuss Maddux's- like representations for $\QRA$, extended to $\CA^{\uparrow}$ by Sagi,
in connection to forcing in set theory.}
\end{myabstract}

\section{Quasi-projective relation algebras}

The pairing technique due to Alfred Tarski, and substantially generalized by Istvan N\;emeti, consists of  defining 
a pair of quasi-projections.
$p_0$ and $p_1$ 
so that in a model $\cal M$ say of  a certain sentence $\pi$, where $\pi$ is built out of these quasi-projections, $p_0$ and $p_1$
are functions and for any element $a,b\in {\cal M}$, there is a $c$ 
such that $p_0$ and $p_1$
map $c$ to $a$ and $b,$ respectively. 
We can think of $c$ as representing the ordered pair $(a,b)$ 
and $p_0$ and $p_1$ are the functions that project the ordered pair onto 
its first and second coordinates.  

Such a technique, ever since introduced by Tarski, to formalize, and indeed succesfully so, set theory, in the calculas of relations 
manifested itself in several re-incarnations in the literature some of which are quite subtle and 
sophisticated.
One is Simon's proof of the representability of quasi-relation algebras $\QRA$ 
(relation algebrs with quasi projections) using a neat embedding theorem for cylindric algebras \cite{Andras}. 
The proof consists of stimulating a neat embeding theorem
via the quasi-projections, in short it is actually a 
{\it a completeness proof}. The idea implemented  is that quasi-projections, on the one hand, generate extra dimensions, and on the other it has 
control over such a stretching. The latter property does not come across very much in Simon's proof, but below we will give an exact rigorous 
meaning to such property. This method can is used by Simon to
apply a Henkin completeness construction. We shall use Simon's technique to further show that $\QRA$ has the superamalgamation property; 
this is utterly unsurprising because Henkin constructions also prove interpolation theorems. This is the case, e.g. 
for first order logics and several of its non-trivial extensions arising from the process of algebraising first order logic, 
by dropping the condition of local finiteness reflecting the fact
that formulas contain only finitely many (free) variables. A striking example in this connection is the algebras studied by Sain and Sayed Ahmed 
\cite{Sain}, \cite{Sayed}.

This last condition is  unwarrented from the algebraic point of view, because it presents an equational formalism of firs order logic.

The view, of capturing extra dimensions, using also quai-projections comes along also very much so, 
in N\'emetis directed cylindric algebras (introduced as a $\CA$ counterpart of 
$\QRA$). In those, S\'agi defined quasi-projections also to achieve a completeness theorem for higher order logics.
The technique used is similar to Maddux's proof of representation of ${\QRA}$s, which further emphasizes the correlation.
We start with making the notion of extra dimensions explicit. We formulate its dual notion, that of compressing dimensions, known as taking neat reducts.
The definition of neat reducts in the standard definition adopted by Henkin, Monk and Tarski in their mongraph, 
deals only with the latter case, but it proves useful to stretch the definition a little allowing 
arbitary substs of $\alpha$ not just initial segments.

\begin{definition} Let ${C}\in \CA_{\alpha}$ and $I\subseteq \alpha$, and let $\beta$ be the order type of $I$. Then
$$Nr_IC=\{x\in C: c_ix=x \textrm{ for all } i\in \alpha\sim I\}.$$
$$\Nr_{I}{\C}=(Nr_IC, +, \cdot ,-, 0,1, c_{\rho_i}, d_{\rho_i,\rho_j})_{i,j<\beta},$$
where $\beta$ is the unique order preserving one-to-one map from $\beta$ onto $I$, and all the operations 
are the restrictions of the corresponding operations on $C$. When $I=\{i_0,\ldots i_{k-1}\}$ 
we write $\Nr_{i_0,\ldots i_{k-1}}\C$. If $I$ is an initial segment of $\alpha$, $\beta$ say, we write $\Nr_{\beta}\C$.
\end{definition}
Similar to taking the $n$ neat reduct of a $\CA$, $\A$ in a higher dimension, is taking its $\Ra$ reduct, its relation algebra reduct.
This has unverse consisting of the $2$ dimensional elements of $\A$, and composition and converse are defined using one spare dimension.
A slight generalization, modulo a reshufflig of the indicies: 

\begin{definition}\label{RA} For $n\geq 3$, the relation algebra reduct of $\C\in \CA_n$ is the algebra
$$\Ra\C=(Nr_{n-2, n-1}C, +, \cdot,  1, ;, \breve{}, 1').$$ 
where $1'=d_{n-2,n-1}$, $\breve{x}=s_{n-1}^0s_{n-1}^{n-2}s_0^{n-1}x$ and $x;y=c_0(s_0^{n-1}x. s_0^{n-2}y)$. 
Here $s_i^j(x)=c_i(x\cdot d_{ij})$ when $i\neq q$ and $s_i^i(x)=x.$
\end{definition}
But what is not obvious at all is that an $\RA$ has a $\CA_n$ reduct for $n\geq 3$. 
But Simon showed that certain relations algebras do; namely the $\QRA$s.

\begin{definition} A relation algebra $\B$ is a $\QRA$ 
if there are elements $p,q$ in $\B$ satisfying the following equations:
\begin{enumarab}
\item $\breve{p};p\leq 1', q; q\leq 1;$
\item  $\breve{p};q=1.$
\end{enumarab}
\end{definition}
In this case we say that $\B$ is a $\QRA$  with quasi-projections $p$ and $q$. 
To construct cylindric algebras of higher dimensions 'sitting' in a $\QRA$, 
we need to define certain terms. seemingly rather complicated, their intuitive meaning 
is not so hard to grasp.
\begin{definition} Let $x\in\B\in \RA$, then , we nee$dom(x)=1';(x;\breve{x})$ and $ran(x)=1';(\breve{x}; x)$, $x^0=1'$, $x^{n+1}=x^n;x$. $x$ 

is a functional element if $x;\breve{x}\leq 1'$.
\end{definition}
Given a $\QRA$, which we denote by $\bold Q$, we have quasi-projections $p$ and $q$ as mentioned above. 
Next we define certain terms in ${\bf Q}$, cf. \cite{Andras}:


$$\epsilon^{n}=dom q^{n-1},$$
$$\pi_i^n=\epsilon^{n};q^i;p,  i<n-1, \pi_{n-1}^{(n)}=q^{n-1},$$
$$ \xi^{(n)}=\pi_i^{(n)}; \pi_i^{(n)},$$
$$ t_i^{(n)}=\prod_{i\neq j<n}\xi_j^{(n)}, t^{(n)}=\prod_{j<n}\xi_j^{(n)},$$
$$ c_i^{(n)}x=x;t_i^{(n)},$$
$$ d_{ij}^{(n)}=1;(\pi_i^{(n)}.\pi_j^{(n)}),$$
$$ 1^{(n)}=1;\epsilon^{(n)}.$$
and let
$$\B_n=(B_n, +, \cdot, -, 0,1^{(n)}, c_i^{(n)}, d_{ij}^{(n)})_{i,j<n},$$
where $B_n=\{x\in B: x=1;x; t^{(n)}\}.$
The intuitive meaning of those terms is explained in \cite{Andras}, right after their definition on p. 271.

\begin{theorem} Let $n>1$
\begin{enumerate} 
\item Then ${\B}_n$ is closed under the operations.
\item ${\B}_n$ is a $\CA_n$.
\end{enumerate}
\end{theorem}
\begin{proof} (1) is  proved in \cite{Andras} lemma 3.4 p.273-275 where the terms are definable in a $\QRA$. 
That it is a $\CA_n$ can be proved as \cite{Andras}  theorem 3.9.
\end{proof} 

\begin{definition} Consider the following terms.
$$suc (x)=1; (\breve{p}; x; \breve{q})$$
and
$$pred(x)=\breve{p}; ranx; q.$$ 
\end{definition}
It is proved in \cite{Andras} that $\B_n$ neatly embeds into $\B_{n+1}$ via $succ$. The successor function thus codes 
extra dimensions. The thing to observe here is that  we will see that $pred$; its inverse; 
guarantees a condition of commutativity of two operations: forming neat reducts and forming subalgebras;
it does not make a difference which operation we implement first, as long as we implement both one after the other.
So the function $succ$ {\it captures the extra dimensions added.}. From the point of view of {\it definability} it says 
that terms definable in extra dimensions add nothing, they are already term definable.
And this indeed is a definability condition, that will eventually lead to stong interpolation property we wnat.

\begin{theorem}\label{neat} Let $n\geq 3$. Then $succ: {\B}_n\to \{a\in {\B}_{n+1}: c_0a=a\}$ 
is an isomorphism into a generalized neat reduct of ${\B}_{n+1}$.
Strengthening the condition of surjectivity,  for all $X\subseteq \B_n$, $n\geq 3$, we have (*)
$$succ(\Sg^{\B_n}X)\cong \Nr_{1,2,\ldots, n}\Sg^{\B_{n+1}}succ(X).$$
\end{theorem}

\begin{proof} The operations are respected by \cite{Andras} theorem 5.1. 
The last condition follows  because of the presence of the 
functional element $pred$, since we have $suc(pred x)=x$ and $pred(sucx)=x$, when $c_0x=x$, \cite{Andras} 
lemmas 4.6-4.10. 
\end{proof}
\begin{theorem}
Let $n\geq 3$. Let ${\C}_n$ be the algebra obtained from ${\B}_n$ by reshuffling the indices as follows; 
set $c_0^{{\C}_n}=c_n^{{\B}_n}$ and $c_n^{{\C}_n}=c_0^{{\cal B}_n}$. Then ${\C}_n$ is a cylindric algebra,
and $suc: {\C}_n\to \Nr_n{\C}_{n+1}$ is an isomorphism for all $n$. 
Furthermore, for all $X\subseteq \C_n$ we have
$$suc(\Sg^{\C_n}X)\cong \Nr_n\Sg^{\C_{n+1}}suc(X).$$ 
\end{theorem}

\begin{proof} immediate from \ref{neat}
\end{proof}  
\begin{theorem} Let ${\C}_n$ be as above. Then $succ^{m}:{\C_n}\to \Nr_n\C_m$ is an isomophism, such that 
for all $X\subseteq A$, we have
$$suc^{m}(\Sg^{\C_n}X)=\Nr_n\Sg^{\C_m}suc^{n-1}(X).$$
\end{theorem} 
\begin{proof} By induction on $n$.
\end{proof}
Now we want to neatly embed our $\QRA$ in $\omega$ extra dimensions. At the same we do not want to lose, our control over the streching;
we still need the commutativing of taking, now $\Ra$  reducts with forming subalgebras; we call this property the $\Ra S$ property.
To construct the big $\omega$ dimensional algebra, we use a standard ultraproduct construction.
So here we go.
For $n\geq 3$, let  ${\C}_n^+$ be an algebra obtained by adding $c_i$ and $d_{ij}$'s for $\omega>i,j\geq n$ arbitrarity and with 
$\Rd_n^+\C_{n^+}={\B}_n$. Let ${\C}=\prod_{n\geq 3} {\C}_n^+/G$, where $G$ is a non-principal ultrafilter
on $\omega$. 
In our next theorem, we show that the algebra $\A$ can be neatly embedded in a locally finite algebra $\omega$ dimensional algebra
and we retain our $\Ra S$ property. 

\begin{theorem} Let $$i: {\A}\to \Ra\C$$
be defined by
$$x\mapsto (x,  suc(x),\ldots suc^{n-1}(x),\dots n\geq 3, x\in B_n)/G.$$ 
Then $i$ is an embedding ,
and for any $X\subseteq A$, we have 
$$i(\Sg^{\A}X)=\Ra\Sg^{\C}i(X).$$
\end{theorem}
\begin{proof} The idea is that if this does not happen, then it will not happen in a fnite reduct, and this impossible \cite{Sayed}.

\end{proof}

\begin{theorem} Let $\bold Q\in {\RA}$. Then for all $n\geq 4$, there exists a unique 
$\A\in S\Nr_3\CA_n$ such that $\bold Q=\Ra\A$, 
such that for all $X\subseteq A$, $\Sg^{\bf Q}X=\Ra\Sg^{\A}X.$
\end{theorem}
\begin{proof} This follows from the previous theorem together with $\Ra S$ property.
\end{proof}

\begin{corollary} Assume that $Q=\Ra\A\cong \Ra\B$ then this lifts to an isomorphism from $\A$ to $\B$.
\end{corollary}
The previous theorem says that $\Ra$ as a functor establishes an equivalence between ${\QRA}$ 
and a reflective subcategory of $\Lf_{\omega}$
We say that $\A$ is the $\omega$ dilation of ${\bf Q}$.
Now we are ready for:

\begin{theorem} $\QRA$ has $SUPAP$.
\end{theorem}
\begin{proof}  We form the unique dilatons of the given algebras required to be superamalgamated. 
These are locally finite so we can find a superamalgam $\D$. Then $\Ra\D$ will be required superamalgam; it contains quasiprojections because the base algebras 
does.
Let $\A,\B\in \QRA$. Let $f:\C\to \A$ and $g:\C\to \B$ be injective homomorphisms .
Then there exist $\A^+, \B^+, \C^+\in \CA_{\alpha+\omega}$, $e_A:\A\to \Ra{\alpha}\A^+$ 
$e_B:\B\to  \Ra\B^+$ and $e_C:\C\to \Ra\C^+$.
We can assume, without loss,  that $\Sg^{\A^+}e_A(A)=\A^+$ and similarly for $\B^+$ and $\C^+$.
Let $f(C)^+=\Sg^{\A^+}e_A(f(C))$ and $g(C)^+=\Sg^{\B^+}e_B(g(C)).$
Since $\C$ has $UNEP$, there exist $\bar{f}:\C^+\to f(C)^+$ and $\bar{g}:\C^+\to g(C)^+$ such that 
$(e_A\upharpoonright f(C))\circ f=\bar{f}\circ e_C$ and $(e_B\upharpoonright g(C))\circ g=\bar{g}\circ e_C$. Both $\bar{f}$ and $\bar{g}$ are 
monomorphisms.
Now $Lf_{\omega}$ has $SUPAP$, hence there is a $\D^+$ in $K$ and $k:\A^+\to \D^+$ and $h:\B^+\to \D^+$ such that
$k\circ \bar{f}=h\circ \bar{g}$. $k$ and $h$ are also monomorphisms. Then $k\circ e_A:\A\to \Ra\D^+$ and
$h\circ e_B:\B\to \Ra\D^+$ are one to one and
$k\circ e_A \circ f=h\circ e_B\circ g$.
Let $\D=\Ra\D^+$. Then we obtained $\D\in \QRA$ 
and $m:\A\to \D$ $n:\B\to \D$
such that $m\circ f=n\circ g$.
Here $m=k\circ e_A$ and $n=h\circ e_B$. 
Denote $k$ by $m^+$ and $h$ by $n^+$.
Now suppose that $\C$ has $NS$. We further want to show that if $m(a) \leq n(b)$, 
for $a\in A$ and $b\in B$, then there exists $t \in C$ 
such that $ a \leq f(t)$ and $g(t) \leq b$.
So let $a$ and $b$ be as indicated. 
We have  $(m^+ \circ e_A)(a) \leq (n^+ \circ e_B)(b),$ so
$m^+ ( e_A(a)) \leq n^+ ( e_B(b)).$
Since $K$ has $SUPAP$, there exist $z \in C^+$ such that $e_A(a) \leq \bar{f}(z)$ and
$\bar{g}(z) \leq e_B(b)$.
Let $\Gamma = \Delta z \sim \alpha$ and $z' =
{\sf c}_{(\Gamma)}z$. (Note that $\Gamma$ is finite.) So, we obtain that 
$e_A({\sf c}_{(\Gamma)}a) \leq \bar{f}({\sf c}_{(\Gamma)}z)~~ \textrm{and} ~~ \bar{g}({\sf c}_{(\Gamma)}z) \leq
e_B({\sf c}_{(\Gamma)}b).$ It follows that $e_A(a) \leq \bar{f}(z')~~\textrm{and} ~~ \bar{g}(z') \leq e_B(b).$ Now by hypothesis
$$z' \in \Ra\C^+ = \Sg^{\Ra\C^+} (e_C(C)) = e_C(C).$$ 
So, there exists $t \in C$ with $ z' = e_C(t)$. Then we get
$e_A(a) \leq \bar{f}(e_C(t))$ and $\bar{g}(e_C(t)) \leq e_B(b).$ It follows that $e_A(a) \leq (e_A \circ f)(t)$ and 
$(e_B \circ g)(t) \leq
e_B(b).$ Hence, $ a \leq f(t)$ and $g(t) \leq b.$
We are done.
\end{proof} 

One can prove the theorem using the dimension restricted free algebra $B=\Fr_1^{\rho}\CA_{\omega}$, where $\rho(0)=2$.
This corresponds to a countable first order language with a sequence of variables of order type $\omega$ and one binary relation.
The idea is that $\Fr_1\QRA\cong  \Ra\Fr_1^{\rho}\CA_{\omega}$. So let 
$a, b\in \Fr_1\QRA$ be such that $a\leq b$. Then there exists $y\in \Sg^{\B}\{x\}$ were $x$ is the free generator
of both, such that
$a\leq y\leq b$.

But we need to show that pairing functions can be defined in $\Ra\Fr_{1}\CA_{\omega}$
We have one binary relation $E$ in our langauge; for convenience, 
we write $x\in y$ instead of $E(x,y)$, to remind ourselves that we are actually working in the language
of set theory. 
We define certain formulas culminating in formulating the axioms of a finitely undecidable theory, better known as Robinson's arithmetic 
in our language. These formulas are taken from N\'emeti \cite{Nemeti}. (This is not the only way to define quasi-projections)
We need to define, the quasi projections. Quoting Andr\'eka and N\'emeti in \cite{AN}, we do this by 'brute force'.

$$x=\{y\}=:y\in x\land (\forall z)(z\in x\implies z=y)$$
$$\{x\}\in y=:\exists z(z=\{x\}\land z\in y)$$
$$x=\{\{y\}\}=:\exists z(z=\{y\}\land x=\{z\})$$
$$x\in \cup y:=\exists z(x\in z\land z\in y)$$
$$pair(x)=:\exists y[\{y\}\in x\land (\forall z)(\{z\}\in x\to z=y)]\land \forall zy[(z
\in \cup x\land \{z\}\notin x\land$$
$$y\in \cup x\land \{y\}\notin x\to z=y]\land \forall z\in x\exists y
(y\in z).$$
Now we define the pairing functions:
$$p_0(x,y)=:pair(x)\land \{y\}\in x$$
$$p_1(x,y)=:pair(x)\land [x=\{\{y\}\}\lor (\{y\}\notin x\land y\in \cup x)].$$
$p_0(x,y)$ and $p_1(x,y)$ are defined.

\section{ Pairing functions in N\'emetis directed $\CA$s}

We recall the definition of what is called weakly higher order cylindric algebras, or directed cylindric algebras invented by N\'emeti 
and further studied by S\'agi and Simon. 
Weakly higher order cylindric algebras are natural expansions of cylindric algebras. 
They have extra operations that correspond to a certain kind of bounded existential 
quantification along a binary relation $R$. The relation $R$ is best thought of as the `element of relation' in a model of some set theory.
It is an abstraction of the membership relation. These cylindric-like algebras 
are the cylindric counterpart of quasi-projective relation algebras, introduced by Tarski. These algebras were 
studied by many authors
including Andr\'eka, Givant, N\'emeti, Maddux, S\'agi, Simon, and others. The reference \cite{Andras} is recommended for other references in the topic.
It also has reincarnations in Computer Science literature 
under the name of Fork algebras.
We start by recalling the concrete versions of directed cylindric algebras:

\begin{definition}(P--structures and extensional structures.) \\
Let $U$ be a set and let $R$ be a binary relation on $U$. The structure
$\langle U; R \rangle$ is defined to be a P--structure\footnote{``P'' stands for ``pairing'' or ``pairable''.} iff for every
elements $a,b \in U$ there exists an element $c \in U$ such that $R(d,c)$ is
equivalent with $d=a$ or $d=b$ (where $d \in U$ is arbitrary) , that is, \\
\\
\centerline{ $\langle U; R \rangle \models (\forall x,y)(\exists z)(\forall w)( R(w,z) \Leftrightarrow (w=x$ or $w=y))$.} \\
\\
The structure $\langle U; R \rangle $ is defined to be a \underline{weak P--structure} iff \\
\\
\centerline{ $ \langle U; R \rangle \models (\forall x,y)(\exists z)(R(x,z) $ and $ R(y,z))$.} \\
\\
The structure $\langle U; R \rangle$ is defined to be {extensional}
iff every two points $a,b \in U$ coincide whenever they have the same
``$R$--children'', that is, \\
\\
\centerline{ $\langle U; R \rangle \models (\forall x,y)(((\forall z) R(z,x) \Leftrightarrow R(z,y)) \Rightarrow x=y) $.}
\end{definition}

\noindent
We will see that if $\langle U; R \rangle$ is a P--structure then one can
``code'' pairs of elements of $U$ by a single element of $U$ and whenever
$\langle U; R \rangle$ is extensional then this coding is ``unique''. In fact,
in $\RCA_{3}^{\uparrow}$ (see the definition below) one can define terms similar
to quasi--projections and, as with the class of $\QRA$'s, one can equivalently
formalize many theories of first order logic as equational theories of certain
$\RCA_{3}^{\uparrow}$'s. Therefore $\RCA_{3}^{\uparrow}$ is in our main interest.
$\RCA_{\alpha}^{\uparrow}$ for bigger $\alpha$'s behave in the same way, an
explanation of this can be found in \cite{Sagi} and can be deduced from our proof, which shows that $\RCA_{3}^{\uparrow}$ has implicitly $\omega$ 
extra dimensions.
\begin{definition}
\label{canyildef}
(${\sf Cs}^{\uparrow}_{\alpha}$, $\RCA^{\uparrow}_{\alpha}$.) \\
Let $\alpha$ be an ordinal. Let $U$ be a set and let $R$ be a binary relation on $U$
such that $\langle U; R \rangle$ is a weak P--structure.
Then the
{full w--directed cylindric set algebra} of dimension $\alpha$ with base
structure $\langle U; R \rangle$ is the algebra: \\
\\
\centerline{$\langle {\cal P}({}^{\alpha}U); \cap, -, C_{i}^{\uparrow(R)}, C_{i}^{\downarrow(R)}, D_{i,j}^{U} \rangle_{i,j \in \alpha}$,} \\
\\
where $\cap$ and $-$ are set theoretical intersection and complementation (w.r.t. ${}^{\alpha}U$),
respectively, $D^{U}_{i,j} = \{ s \in {}^{\alpha}U: s_{i}=s_{j} \}$ and
$ C_{i}^{\uparrow(R)}, C_{i}^{\downarrow(R)}$ are defined as follows. For every
$X \in {\cal P}({}^{\alpha}U)$: \\
\\
\indent $ C_{i}^{\uparrow(R)}(X) = \{ s \in {}^{\alpha}U: (\exists z \in X)( R(z_{i},s_{i})$ and $(\forall j \in \alpha)(j \not=i \Rightarrow s_{j}=z_{j})) \},$ \\
\indent $ C_{i}^{\downarrow(R)}(X) = \{ s \in {}^{\alpha}U: (\exists z \in X)( R(s_{i},z_{i})$ and $(\forall j \in \alpha)(j \not=i \Rightarrow s_{j}=z_{j})) \}.$ \\
\\
The class of {w--directed cylindric set algebras} of dimension $\alpha$
and the class of {directed cylindric set algebras} of dimension $\alpha$
are defined as follows. \\
\\
\centerline{$ w-{\sf Cs}^{\uparrow}_{\alpha} ={\bf S} \{ {\cal A}: \ {\cal A}$ is a full
w--directed cylindric set algebra of dimension $\alpha$} \\
\centerline{ \indent \indent \indent \indent with base structure $\langle U; R \rangle$,
for some weak P--structure $\langle U; R \rangle \}$.} \\
\\
\centerline{$ {\sf Cs}^{\uparrow}_{\alpha} ={\bf S} \{ {\cal A}: \ {\cal A}$ is a full
w--directed cylindric set algebra of dimension $\alpha$} \\
\centerline{ \indent \indent \indent \indent with base structure $\langle U; R \rangle$,
for some extensional P--structure $\langle U; R \rangle \}$.} \\
\\
The class $\RCA^{\uparrow}_{\alpha}$ of {representable directed cylindric algebras} of
dimension $\alpha$ is defined to be $\RCA^{\uparrow}_{\alpha} = {\bf SP}{\sf Cs}^{\uparrow}_{\alpha}$.
\end{definition}

The main result of Sagi in \cite{Sagi} is a direct proof for the following: \\
\\
\begin{theorem}\label{rep}
{\em $\RCA^{\uparrow}_{\alpha}$ is a finitely axiomatizable
variety whenever $\alpha \geq 3$ and $\alpha$ is finite}
\end{theorem}

$\CA^{\uparrow}_3$ denotes the variety of directed cylindric algebras of dimension $3$ 
as defined in \cite{Sagi} definition 3.9. In \cite{Sagi}, it is proved that
$\CA^{\uparrow}_3=\RCA^{\uparrow}_3.$ A set of axioms is formulated on p. 868 in \cite{Sagi}. 
Let $\A\in \CA^{\uparrow}_3$. 
Then we have quasi-projections 
$p,q$ defined on $\A$ as defined in \cite{Sagi} p. 878, 879. We recall their definition, which is a little bit complicated because 
they are defined as formulas in the corresponding second 
order logic.
Let $\cal L$ denote the untyped logic corresponding to directed $\CA_3$'s as  defined p.876-877 in \cite{Sagi}. It has only $3$ variables. 
There is a correspondance between formulas (or formual schemes)  in this language and $\CA^{\uparrow}_3$ terms. 
This is completely analgous to the corresponance between $\RCA_n$ terms and first order formulas containing only $n$ variables.
For example $v_i=v_j$ corresponds to $d_{ij}$, $\exists^{\uparrow} v_i(v_i=v_j)$ correspond to ${\sf c}^{\uparrow}_i d_{ij}$.
In \cite{Sagi} the following formulas (terms) are defined:

\begin{definition} Let $i,j,k \in 3$ distinct elements. 
We define variable--free $RCA^{\uparrow}_{3}$ terms as follows:
\begin{tabbing}
\indent \= $ v_{i} = \{ \{ v_{j} \}_{R} \}_{R}$ \ \ \= is \indent \= $\exists v_{k}( v_{k} = \{ v_{j} \}_{R} \wedge v_{i} = \{ v_{k} \}_{R})$, \kill
\> $ v_{i} \in_R v_{j}$ \indent \> is \> $\exists^{\uparrow}v_{j}(v_{i}=v_{j})$, \\
\> $ v_{i} = \{ v_{j} \}_{R}$ \> is \> $\forall v_{k}( v_{k} \in_R v_{j} \Leftrightarrow v_{k}=v_{j}) $, \\
\> $ \{ v_{i} \}_{R} \in_R v_{j}$ \> is \> $\exists v_{k}( v_{k} \in_R v_{j} \wedge v_{k} = \{ v_{i} \}_{R})$, \\
\> $ v_{i} = \{ \{ v_{j} \}_{R} \}_{R}$ \> is \> $\exists v_{k}( v_{k} = \{ v_{j} \}_{R} \wedge v_{i} = \{ v_{k} \}_{R})$ ,\\
\> $ v_{i} \in_R \cup v_{j}$ \> is \> $\exists v_{k}(v_{i} \in_R v_{k} \wedge v_{k} \in_R v_{j})$.
\end{tabbing}
\end{definition}

Therefore $pair_{i}$ (a pairing function) can be defined as follows: \\

\indent $\exists v_{j} \forall v_{k}( \{ v_{k} \}_{R} \in_R v_{i} \Leftrightarrow v_{j} = v_{k}) \ \wedge $ \\
\indent $\forall v_{j} \exists v_{k}( v_{j} \in_R v_{i} \Rightarrow v_{k} \in_R v_{j} ) \ \wedge $ \\
\indent $\forall v_{j} \forall v_{k}( v_{j} \in_R \cup v_{i} \ \wedge \ \{ v_{j} \} \not\in_{R} v_{i} \ \wedge \ v_{k} \in_R \cup v_{i} \ \wedge \ \{ v_{k} \} \not\in_{R} v_{i} \Rightarrow v_{j} = v_{k} )$. \\

It is clear that this is a term built up of diagonal elements and directed cylindrifications. 
The first quasi-projection  $v_{i} = P(v_{j})$ can be chosen as: \\
\\
\indent $pair_{j} \ \wedge \ \forall^{\downarrow} v_{j} \exists^{\downarrow} v_{j} (v_{i} = v_{j})$. \\
\\
and the second quasiprojection  $v_{i} = Q(v_{j})$ can be chosen as: \\
\\
\centerline{$pair_{j} \ \wedge \ (( \forall v_{i} \forall v_{k} ( v_{i} \in_R v_{j} \ \wedge \ v_{k} \in_R v_{j}  \Rightarrow v_{i} = v_{k} )) \Rightarrow v_{i} = P(v_{j})) \ \wedge $} \\
\centerline{ $(\exists v_{i} \exists v_{k}( v_{i} \in_R v_{j} \ \wedge \ v_{k} \in_R v_{j} \ \wedge \  v_{i} \not= v_{k}) \Rightarrow (v_{i} \not= P(v_{j}) \ \wedge \ \exists^{\downarrow} v_{j} \exists^{\downarrow} v_{j}(v_{i} = v_{j})))$.}

\begin{theorem}
Let $\B$ be the relation algebra reduct of $\A$; then $\B$ is a relation algebra, and the 
variable free terms corresponding to the formulas  $v_i=P(v_j)$ and $v_j=Q(v_j)$ 
call  them $p$ and $q$, respectively, are quasi-projections.
\end{theorem}
\begin{proof} 
One proof is very tedious, though routine. One translates the functions as variable free terms in the language of $\CA_3$ and use 
the definition of composition and converse in the $\RA$ reduct, to verify that they are quasi-projections.
Else one can look at their meanings on set algebras, which we recall from Sagi \cite{Sagi}.
Given a cylindric set algebra $\cal A$ with base $U$ and accessibility relation $R$
$$(v_i=P(v_j))^A=\{s\in {}^3U: (\exists a,b\in U)(s_j=(a,b)_R, s_i=a\}$$
$$(v_i=Q(v_j)^A=\{s\in {}^3U: (\exists a,b\in U)(s_j=(a,b)_R, s_i=b\}.$$
First $P$ and $Q$ are functions, so they are functional elements. 
Then it is clear that in this set algebras that $P$ and $Q$ are quasi-projections. 
Since $\RCA^{\uparrow}_3$ is the variety generated by set algebras, they have the same meaning in the class $\CA^{\uparrow}_3.$
\end{proof}

Now we can turn the class around. Given a $\QRA$ one can define a directed $\CA_n$, for every finite $n\geq 2$.
This definition is given by N\'emeti and Simon in \cite{NS}. 
It is vey similar to Simon's definition above (defining $\CA$ reducts in a $\QRA$, 
except that directed cylindrifiers along a relation $R$ are 
implemented.

\begin{theorem} The concrete category $\QRA$ 
with morphisms injective homomorphisms, and that of $\CA^{\uparrow}$  with morphisms also injective homomorphisms are equivalent.
in particular $\CA^{\uparrow}$ of dimension $3$ is equivalent to $\CA^{\uparrow}$ for $n\geq 3$.
\end{theorem}
\begin{proof} Given $\A$ in $\QRA$ we can associte a directed $\CA_3$, homomorphism are restrictions and vice versa; these are inverse Functors.
However, when we pass from an $\QRA$ to a $\CA^{\uparrow}$ and then take the $\QRA$ reduct, 
we may not get back exactly to the $\QRA$ we started off with,
but the new quasi projections are definable from the old ones. 
Via this equivalence, we readily  conclude that $\RCA_3\to \RCA_n$ are also equivalent.
\end{proof}

\begin{corollary} The class $\CA^{\uparrow}$ has the super amalgamation property.
\end{corollary}
\begin{proof} The functor from $\QRA$ to $\CA^{\uparrow}$ preserves order.
\end{proof}

\section{Godel's first for finite variable fragments}

This section is a summary of work of N\'emeti \cite{N}, reported in \cite{AN}.

There has been some debate over the impact of G\"odel's incompleteness theorems on Hilbert's Program, 
and whether it was the first or the second incompleteness theorem that delivered the coup de grace. 

Undoubtedly the opinion of those most directly involved in the developments 
were convinced that the theorems did have a decisive impact. 

G\"odel announced the second incompleteness theorem in an abstract published in October 1930: 
no consistency proof of systems such as Principia, Zermelo-Fraenkel set theory, 
or the systems investigated by Ackermann and von Neumann is possible by methods 
which can be formulated in these systems.

G\"odel's theorems have a profound impact Hilbert's program. Through a careful G\"odel 
coding of sequences of symbols (formulas, proofs), G\"odel showed that in theories $T$ which contain a sufficient amount of arithmetic, 
it is possible to produce a formula $Pr(x, y)$ which "says" that $x$ is (the code of) a proof of (the formula with code) $y$. 
Specifically, if $0 = 1$ is the code of the formula $0 = 1$, 
then $ConT = \forall (x \neg Pr(x,0 = 1))$ may be taken to "say" that $T$ is consistent (no number is the code of a derivation in $T$ of $0 = 1$). 
The second incompleteness theorem $(G2)$ says that under certain assumptions about $T$and the coding apparatus, $T$ does not prove $ConT$. 

This shattered Hilbert's hopes of proving that set theory is consistent, by finitary means, presumably formalizable 
in set theory (it is hard to visualize 'finitary means" that {\it is not} formalizable in set theory, or even Peano arithmetc).  
This means that mathematicians will be always threatened that one day, 
some mathematician, or rather set-theoretician,  will find an inconsistency. Nevertheless, 
with the amount of research done in set theory, in the last decades, deems this possibility as far fetched, and some mathematicians go 
as far as to say impossible. This is a fair view, if there were a consistency we would have probably stumbled upon it by now.

In the above cited results, the ideas are not too difficult, but implementing the details is highly technical and complicated.
N\'emeti generalized Godel's first theorem as follows:

\begin{theorem}\label{n} 
\begin{enumarab}

\item There is a computable, 
structural translation $\tr: L_{\omega}\to L_3(E, 2)$ such that $\tr$ has a recursive image and the following are 
true for all sets of sentences $Th\cup \{\phi\}$ in $L_{\omega}$

(a) $Th\models \phi\longleftrightarrow \tr(Th)\vdash_n\tr(\phi).$

(b) $Th\models \phi\longleftrightarrow \tr(Th)\models \tr(\phi).$

\item There is a computable, structural translation function
$\tr: L_{\omega}(E,2)\to L(E, 2)$ such that 
$\tr$ has a recursive range and the following (c) and (d) are true

(c) Statements (a) and (b) above hold and $Th\models \neg \tr(\bot).$ 
Furthermore, $ZF\models \neg\tr(\bot).$

(d) $\neg\tr(\bot) \models \phi\longleftrightarrow \tr(\phi)$ 

\end{enumarab}
\end{theorem}
Using this translation map he proves:
\begin{theorem}
There is a formula $\psi\in L_3$ such that no consistent 
recursive extension $T$ of $\psi$ is complete, and moreover, 
no recursive extension of $\psi$ separates the $\vdash$ consequences
of $\psi$ from the $\psi$ refutable sentences.
\end{theorem} 
\begin{proof} We give a sketch of proof for $L_4$. This is implicit in the Tarski Givant approach, 
when they interpreted $ZF$ in $RA$. $L_4$ is very close to $RA$ but not quite
$RA$, it s a little bit stronger. The technique is called the {\it pairing} technique, which uses quasi projections to code extra variable, 
establishing the completeness theorem
above for $\vdash_n$. 

We have one binary relation $E$ in our langauge; for convenience, 
we write $x\in y$ instead of $E(x,y)$, to remind ourselves that we are actually working in the language
of set theory. 
We define certain formulas culminating in formulating the axioms of a finite undecidacle theory, better known as Robinson's arithmetic 
in our language. These formulas are taken from N\'emeti.
We need to define, the quasi projections. Quoting Andr\'eka and N\'emeti, we do this by 'brute force'.
We now formulate the desired $\lambda$.

Havng defined the pairs,  we g on as follows:

$$x\in Ord= : \text {`` $x$ is an ordinal, i.e. $x$ is transitive and $\in$ is a total ordering on $x$},$$
$$x\in Ford=:x\in Ord\land \text { ``every element of $x$ is a successor ordinal "}$$
$$\text { i.e. $x$ is a finite ordinal }.$$
$$x=0=: ``x\text { has no element }"$$
$$sx=z=:z=x\cup \{x\},$$
$$x\leq y=:x\subseteq y,$$
$$x<y=: x\leq y\land x\neq y,$$
$$x+y=z=:\exists v(z=x\cup v\land x\cap v=0 \land$$
$$\text {``there exists a bijection between $v$ and $y$"})$$
$$x\cdot y=z=:\text { ``there is a bijection between $z$ and $x\times y$}"$$
$$x\underline{exp} y=z: \text { there is a bijection between 
$z$ and the set of all functions from $y$ to $x$}"$$
Now $\lambda$' is the formula saying that:
$0, s, +, \cdot, \underline{exp}$ are functions of arities $0,1,2,2,2$ on $Ford$
and 
$$(\forall xy\in Ford)[sx\neq 0\land sx=sy\to x=y)\land (x<sy\longleftrightarrow x\leq y)
\land$$
$$\neg(x<0)\land (x<y\lor x=y\lor y<x)\land (x+0=x)\land (x+sy=s(x+y))\land (x.0=0)$$
$$\land (x\cdot sy=x\cdot y+x)\land (x\underline{exp} 0=s0)\land (x\underline{exp} sy=x\underline{exp} y\cdot x)].$$
Now the existence of the desired incompletable $\lambda$ readily follows:
$\lambda\in Fm_{\omega}^0$.
Let $p=r(p_0(x,y))$ and $q=r(p_1(x,y))$ be the pairing functions as defined above,. 
where $r$ be the recursive function mapping $Fm_3^2$ into $RAT.$ (It is not hard to construct such an function, that also preserves meaning).
$$\pi_{RA}=(p^{\lor};p\to Id)\cdot (q^{\lor};q\to Id)\cdot (p^{\lor};q).$$
Then $\pi_{RA}\in RAT$ since $p_i(x,y)\in Fm_3^2.$
Let $\lambda\in Fm_{\omega}^0$ be inseparable and let
$\eta=(r(\tr(\lambda))\cdot \pi_{RA}$. 
From the definition of $r$ and $f$ we have 
$\eta\in RAT_1$. Let $\Fm_4$ be the algebra of resricted formulas using $4$ variables.
Let ${\cal G}={\cal F}r_1SimRA.$
Let $h: {\cal G} \to \Ra\Fm_4$ be the 
homomorphism that takes the free generator of $\cal G$ to $x\in y.$
Let $\psi=h(\eta)$. Then $\psi\in Fm^{\Lambda_3.}$ 
$\psi$ is the desired formula. (Here we use that the $\Ra$ reduct of a $CA_4$ is a relation algebra.
\end{proof}
The generalization of G\"odel's first theorem, has a very natural algebraic counterpart;  the least that can be said for his second.
The following is slighly new and it depends only on Godel's incompleteness theorem for $L_4$. The free algebras adressed in the next theorem are called dimension restricted free 
algebras.
\begin{corollary}  
\begin{enumroman}
\item Let $\omega\geq m>3$. Let $\beta$ be a cardinal $<\omega$ 
and $\rho:\beta\to \wp(3)$ such that $\rho(i)\geq 2$ for some $i\in \beta.$
Then $\Fr_{\beta}^{\rho}S\Nr_3\CA_m$ is not atomic. 

\item Let $m\geq n>3$ and $\rho:\beta\to \wp(n)$ where $\beta<\omega$ 
and $\rho(i)\geq 2$ for some $i\in \beta$. 
Then $\Fr_{\beta}^{\rho}S\Nr_n\CA_m$ is not atomic.
In particular, $\Fr_{\beta}\CA_4$ and $\Fr_{\beta}\RCA_4$ are not atomic.
\end{enumroman}
\end{corollary}
\begin{corollary}(Maddux) For each finite  $n\geq 3$, The equational theories of $\sf{Df_n}$ and ${\CA_n}$ are undecidable
\end{corollary}
Maddux's proof followed an entirely different route, using the undecidability of the word problem for semigroups.

\section{Godel's second for finite variable fragments}
Our work here is inspired by work of Andreka Madarasz and Nemtii, on working out a Godels second incompleteness theorem for certain strong enough axiomatizations of special relativity.
having a periodic object in their model, the succeed to code $N$, and then the rest follows like the classical case.

We work with $n=3$, and we assume that we have equality. All the results extend to the case when we {\it do not} have equaity 
but we have a tenary relation symbol, instead of a binary one. (This follows from theorem \ref{n}). 

Godel's second theorem follows from the first by formalizing the meta mathematical proof of it into the formal system 
whose consistency is at stake. So  such theories should 
be strong enough to encode the proof of the first incompleteness theorem. Roughly the provability relation $p(x,y)$ ($x$ proves $y$) 
not only proves, when it does it can prove that it proves. given a theory $T$ containing arithmetic, let 
$Prb_T(\sigma)$ denotes $\exists x p(x, \sigma)$.
Formally:

\begin{definition} A theory $T$ is strong enough if when 
$T$ proves $\phi$ then $T$ proves that $T$ proves $\phi$
In more detail,
\begin{enumarab}
\item $T$ contains Robinson's arithmetic

\item for any sentence $\sigma$, $T\vdash \sigma$, then $T\vdash Prb_T(\sigma)$
\item for any sentence $\sigma$, $T\vdash (Prb_T(\sigma)\to Prb_TPrb_T(\sigma)$)
\item For any sentences $\rho$ and $\sigma$, $T\vdash Prb_T(\rho\to \sigma)\to (Prb_T\rho\to Prb_T \sigma).$
\end{enumarab}
\end{definition}
Strong theories are strong enough not to prove their consistency, if they are consistent. 
Robinsons arithmetic is not strong enough but $PA$ and $ZF$ are.
So we need to capture at least $PA$ in $L_n$. 
This will be done in a minute. 
In fact, we can capture the whole of $ZF$, but we will be content only with $PA$, which is sufficient for our process.

Clearly $\psi$ is consistent (we are in $ZF$ set theory). 
Now, we can interpret Robinson's arithmetic $Q$ in our theory 
$\psi$, and this way we can prove all those parts of G\"{o}del's
incompleteness theorems (together with the related theorems like Rosser's)
which hold for $Q$. 

However, we want to establish 
stronger incompleteness results which hold for Peano's Arithmetic $PA$, like for example that $PA$ does not prove $Con(PA)$. 
So far what we have is not enough, to render this form of Godel's second incompleteness theorem. 

$PA$ is stronger than $Q$; because it has the induction schema.
So what strikes one as the obvious  thing to do, is to introduce an axiom schema $Ax(ind)$ which
postulates a natural induction principle for the theory of $\psi$. 


We note that N\'emeti defined $\psi$ in a language with only one binary relation, but the operation 
symbols of Peano arithmetic are definable in $Th(\psi)$ (See above). 
In particular, the successor function $succ$ is definable (This analogous to the the interpretability 
of Peano arithmetic in set theory).

Our work in what follows is inspired and is in fact very close to the work of Andreka et all, when they formalized Godel's second, 
in strong enough first order fragments of special 
relativity.

Now the induction schema has the form $ind(\psi ,x)$ is defined as follows. 

\[
\forall x((\psi (0)\wedge \psi (x
)\rightarrow \psi (\text{suc}(x))\Rightarrow
(\forall x)\psi (x)).
\]

Now,

$${\bf Ax(ind)}:=\{ind(\psi ,x): \psi(x) \text { is a formula using $3$ variables }\}.$$

And we define $T^+$ as follows:
\[
T^+:=\psi + {\bf Ax(ind)}
\]

By definition, $T^+$ is an extension of $\psi$ by a finite schema
of axioms, it is consistent and it is valid in the standard models of $\psi$. 


\begin{theorem} There is a formula $Con(T^+)$
using only $3$ variables, such that in each model $\frak{M}\vDash T^+$ 
this formula expresses the consistency of $T^+$. Furthermore,
$$T^+\nvdash  Con(T^+)$$ 
and 
$$T^+\nvdash \neg Con(T^+).$$
\end{theorem}
\begin{proof} Firstly,  $PA$ can be interpreted in $T^+$
because the axioms of $T^+$ were
chosen in such a way as to make this true.
The axiom system $T^+$ is given by a finite schema,
completely analogous with the axiom system of $PA$. Therefore, the axiom
theory $T^+$ can also be formalized in $PA.$  
Hence in $T^+$, like $PA$, there is formula $pr(x,y)$ 
expressing that $x$ is the G\"{o}del number of a proof from $T^+$
of a formula $\varphi $ of  whose G\"{o}del number is $y$.
Now, $\exists xpr(x,y)$ is a provability
formula $\pi (y)$ which in $T^+$ expresses that $y$ is the G\"{o}%
del number of an $L_3$ formula provable in $T^+$. Furthermore, one can
easily check that the L\"{o}b conditions (as presented, e.g., in [10.
Def.2.16. p.163]) are satisfied by $\pi (y)$ and by $T^+$. Now, we
choose $Con(T^+)$ to be $\urcorner \pi $($False$).
The rest follows the standard proof.
Also, the generalization for (consistent) extensions of $T^{+}$ with
finitely many new axioms can be proved like the classical case; if we have a $%
\sigma_{1}$ definition of the G\"{o}del numbers of the axioms of $T^+$ 
then we can extend this $\sigma_{1}$-definition to ``$T^+$ an
extra (concrete) axiom, say $\varphi $'', since $\varphi $ has a concrete G\"{o}del number $\lceil \varphi \rceil .$
\end{proof}

Our next thorem says that truth in our theory is independent of $ZF$:

\begin{theorem}  There is a formula $\varphi $
using $3$ variables and an extension $T^{++}$ of $T^+$ in $L_3$ such that truth of statement (i)
below is independent of $ZFC$.

\begin{description}
\item  (i) 
\[
T^{++}\vDash \varphi 
\]
\end{description}

\end{theorem}
\begin{proof} Choose $T^{++}$ such that 
$Th(\underline{\omega })$ of full first-order arithmetic can be
interpreted  in $T^{++}$.  In $Th(\underline{\omega })$ there exist a formula,
$\psi $, such that the statement "$\underline{\omega }\vDash \psi $''
is independent of $ZFC$ (assuming $ZFC$ is consistent). Such a $\psi $ is the G%
\"{o}delian formula Con($ZF$),  then ``$\tr(\psi)\in  T^{++}$'' or equivalently ``
$T^{++}\vDash \tr(\psi )$'' is a statement about  $T^{++}$ whose
truth is independent from $ZFC.$ 
\end{proof}

\section{Forcing in relation and cylindric algebras}

Tarski used the theory of relation algebras to express Zermelo-Fraenkel set theory as a system of equations without variables. 
Representations of relation algebras will take us back to set-theoretic relational systems. 

On the other hand, Cohen's method of forcing provides us a way to build new models of set theory and to establish the independence 
of many set-theoretic statements. In \cite{z} a way of building the missing link to connect relation algebras and the method of forcing is presented.
Let ${\bf QRA}$ stand for the class of quasi relation algebras.
Maddux proved using a technique which we call a Maddux style representation, that every ${\bf QRA}$ is representable.

Now, see \cite{z} p.55, theorem 13,

\begin{theorem}
\begin{enumarab}
\item Let $\A$ be a simple countable $\QRA$ 
that is based on a model $(M,\in)$ of set theory. 
Let $h$ be a Maddux style representation of $\A$. If $d\in A$ is well founded relation on $M$, then $h(d)$ is well founded
\item Let $\A$ be a simple countable $\CA^{\uparrow}_3$
that is based on a model $(M, \in)$ of set theory. 
Let $h$ the Sagi represenation. If $R\in A$
is well founded then so $h(d)$.
\end{enumarab}
\end{theorem}

So Maddux's and Sagi's style representations, in fact preverses well foundeness of relations, which is not an elementary fact. 
In Theorem 14,  p. 61 of \cite{z}, a characteriszation of simple ${\bf QRA}$'s  with a distinguished element 
that are  isomorphic to an algebra of relations arising from a countable transitive model of enough set theory is given.

So let $h$ be the Maddux style representation of such an $\A$,  on a set algebra with base $U$. Then $U$ is countable, and $h(e)$ "set like".
By Mostowski Collapsing theorem, there is a transitive $M$ and a one to one 
map $g$ from $U$ onto $M$, such that $g$ 
is an isomorphism betwen $(U, h(e))$ and $(M,\in)$, where
$\in$ is the real membership.
$(M,\in)$ is also, a model of enough set theory. Let $M[G]$ is generic extension of $M$, formed by the methods of forcing,
and take the ${\bf QRA}$, call it $\A[G]$
corresponding to  $(M[G],\in).$  Assume for example that $\A$ models the translation of the continuum hypothesis, 
while $M[G]$ models its negation. 
Then we can conclude that $\A$ and $\A[G]$ are simple countable relation algebras that are equationaly distinct.
similary for the corresponding directed $\CA$s.

One can carry similar investigations in the context of directed cylindric algebras instead of ${\bf QRA}$, 
by noting that representations of such algebras defined by 
Sagi also preserves well foundness.

\end{document}